\documentclass[12pt]{article}
\usepackage{lipsum}

\usepackage[margin=1in]{geometry}
\usepackage{fancyhdr}
\usepackage{graphicx}
\usepackage{mathtools}
\usepackage{amsmath, amssymb, amsfonts, amsbsy, amsthm, latexsym, color}
\usepackage{ulem}
\usepackage{mathtools}
\usepackage{mathrsfs}
\usepackage{cancel}

\usepackage{hyperref} 
\usepackage{empheq}

\usepackage[T1]{fontenc}
\usepackage{authblk}
\usepackage{graphicx}
\usepackage{enumerate}
\usepackage{caption}
\usepackage{subcaption}
\newcommand{\RN}[1]{%
  \textup{\uppercase\expandafter{\romannumeral#1}}%
}
\everymath{\displaystyle}
\usepackage{float}

\usepackage[ruled,vlined]{algorithm2e}
\newtheorem{theorem}{Theorem}[section]
\newtheorem{proposition}[theorem]{Proposition}
\newtheorem{lemma}[theorem]{Lemma}

\theoremstyle{definition}

\newtheorem{example}[theorem]{Example}

\newtheorem{remark}[theorem]{Remark}

\numberwithin{equation}{section}

\usepackage{romannum}
\newcommand{\tl}{\hat\lambda}
\newcommand{\W}{\widetilde}

\begin{document}
\pagenumbering{arabic}
\title{Perturbation of invariant subspaces for ill-conditioned eigensystem}
\author{He Lyu, Rongrong Wang}
\affil{Michigan State University}

%
%
%


\date{}

\maketitle



\begin{abstract}
Given a diagonalizable matrix $A$, we study the stability of its invariant subspaces when its matrix of eigenvectors is ill-conditioned.  Let $\mathcal{X}_1$ be some invariant subspace of $A$ and $X_1$ be the matrix storing the right eigenvectors that spanned $\mathcal{X}_1$. It is generally believed that when the condition number $\kappa_2(X_1)$ gets large,  the corresponding invariant subspace $\mathcal{X}_1$  will become unstable to perturbation. This paper proves that this is not always the case. 
Specifically, we show that the growth of $\kappa_2(X_1)$ alone is not enough to destroy the stability. As a direct application, our result ensures that when $A$ gets closer to a Jordan form, one may still estimate its invariant subspaces from the noisy data stably.


\textbf{MSC Classification Codes.} 47A15, 47A55

\end{abstract}
\section{Introduction}\label{sec:intro}
Let $A\in\mathbb{C}^{n,n}$ be a diagonalizable matrix. An invariant subspace $\mathcal{X}$ of $A$ is one that satisfies
\[
A\mathcal{X}\subseteq \mathcal{X}.
\]
When a small perturbation is added to $A$, its invariant subspace $\mathcal{X}$ will be perturbed accordingly.  The goal of stability analysis is to bound the perturbation of $\mathcal{X}$ in terms of the perturbation of $A$.  

More precisely, any invariant subspace of a diagonalizable matrix is spanned by a subset of the right eigenvectors. Suppose $A=X\Lambda X^{-1}$ is the eigen-decomposition of $A$, and $X$ contains the normalized eigenvectors as columns. 
Suppose $X$ can be partitioned into two blocks $X=[X_1,X_2]$, 
\begin{equation}\label{eq:defination}
A[X_1, X_2] = [X_1, X_2]\begin{pmatrix}
\Lambda_1 & \ \\
\ & \Lambda_2
\end{pmatrix},
\end{equation}
and the corresponding eigenvalues stored in the diagonal matrices $\Lambda_1$ and $\Lambda_2$ are separable. The objective of our study is the perturbation of the invariant subspace $\mathcal{X}_1=span(X_1)$.

The separation between $\Lambda_1$ and $\Lambda_2$ ensures that, for sufficiently small perturbation, the eigen-decomposition of the perturbed matrix $\W A=A+\Delta A$ has a similar block structure,
\begin{equation}\label{eq:tildeA} \W A[\W X_1, \W X_2] = [\W X_1, \W X_2]\begin{pmatrix}
\W \Lambda_1 & \ \\
\ & \W \Lambda_2
\end{pmatrix},\end{equation}
therefore we can use  $\W{\mathcal{X}}_1=span(\W{X}_1)$ as an estimation of $\mathcal{X}_1$. The estimation error is usually measured by the $\sin\Theta$ distance between $\mathcal{\W X}_1$ and $\mathcal{X}_1$. For a fixed perturbation level,  smaller $\sin\Theta$ distances imply better stability. 
 
\subsection{Motivation}\label{sec:motive}
Since the exact $\sin\Theta$ distance is hard to calculate, a central task in eigen-perturbation analysis is establishing useful upper bounds with simpler expressions {\cite{varah1970computing,stewart1990matrix,greenbaum2020first,karow2014perturbation,ipsen2003note,golub1976ill,demmel1986computing,kato2013perturbation,chatelin2011spectral,gohberg2006invariant,davis1970rotation}}. Explicitly, we prefer upper bounds that are expressed by simple quantities related to $A$ and $\Delta A$, such as $\|\Delta A\|$, $\|A\|$, the condition number of $X$, and the gap between  $\Lambda_1$ and $\Lambda_2$, since these quantities are more likely to be given as prior information and/or are easier to estimate when the exact $A$ is unknown. The most well-known bound is perhaps the one given by the Davis-Kahan theorem \cite{davis1970rotation}, which states that for Hermitian matrices, the  $\sin\Theta$ distance  depends only on $||\Delta A||$ and the eigengap. For non-Hermitian matrices, however, it is believed that the stability of $\mathcal{X}_1$ also depends on the condition number of the eigenvector matrix $X$: an ill-conditioned $X$ would cause instabilities of the invariant subspaces. However, a tight relationship between the condition number and the $\sin\Theta$ distance is yet to be established.


In \cite{stewart1973error}, Stewart discovered that the key quantity that determines the stability of invariant subspaces is not the eigengap nor the condition numbers, but a new quantity called separation defined using the norm of a Sylvester operator related to $A$. 
Despite the efficacy in characterizing the subspace stability,  separation is difficult to estimate when the original matrix $A$ is not fully known, which is unfortunately often the case in practice. In contrast, it is much easier to come up with a rough estimate of (or be given some prior knowledge of) the condition number and the eigengap.  Therefore, characterizing the stability directly in terms of  condition numbers and eigengaps is still of great practical interest.

In this paper, we focus on studying the optimal dependence of the stability of invariant subspaces on the condition numbers, i.e., the relation between $\sin\Theta(\mathcal{X}_1,\W{\mathcal{X}}_1)$ and $\kappa_2(X_1)$. 

This problem has been partially studied in several papers {\cite{ipsen2003note,stewart1973error,stewart1990matrix,varah1970computing}}. 
As mentioned in \cite{stewart1990matrix}, stemming from Stewart's  tight upper bound \cite{stewart1973error} of the $\tan\Theta$ angles expressed in terms of the separation, one can use a simple relation \eqref{eq:sep_bound} between the separation and the condition numbers to obtain the following bound  (see the derivation in the appendix)
\begin{equation}\label{eq:plug_stewart}
    \|\tan\Theta(\mathcal{X}_1,\mathcal{\W X}_1)\|<\frac{2\kappa_2(X_1)\kappa_2(V_2)\|\Delta A\|}{[\delta_0-2\kappa_2(X_1)\kappa_2(V_2)\|\Delta A\|]_+},
\end{equation}
where $[\cdot]_+$ is the positive part of the input, $V_2$ is the second block in $(X^{-1})^*=[V_1,V_2]$ ($^*$ refers to the conjugate transpose), and $\delta_0$ is the eigengap defined as \[\delta_0=\max_{t_0\in\mathbb{C}}\left\{\max\left\{\min_{\lambda\in S(\Lambda_1)}|\lambda-t_0|-\max_{\mu\in S(\Lambda_2)}|\mu-t_0|,\min_{\lambda\in S(\Lambda_2)}|\lambda-t_0|-\max_{\mu\in S(\Lambda_1)}|\mu-t_0|\right\}\right\}.
\]
Here $S(\Lambda_i)$, $i=1,2$ are the sets of eigenvalues contained in $\Lambda_i$, $i=1,2$, respectively, and are called the spectral sets. 
Let us provide some intuition of $\delta_0$ to assist understanding.  $\delta_0>0$ essentially means that there exists a disk in the complex plane that separates $S(\Lambda_1)$ from $S(\Lambda_2)$, i.e., for some $t\in\mathbb{C}$ and radius $\rho>0$, the disk $B(t, \rho)$ satisfies either 
\begin{itemize} 
\item[(i)] $S(\Lambda_2) \subseteq B(t, \rho)$ and $S(\Lambda_1) \subseteq \mathbb{C}\backslash B(t, \rho)$; or 
\item[(ii)] $S(\Lambda_1) \subseteq B(t, \rho)$ and $S(\Lambda_2) \subseteq \mathbb{C}\backslash B(t, \rho)$.
\end{itemize}
In other words, one of two spectral sets is completely inside the disk, and the other is completely outside the disk. 
Comparing $\delta_0$ with another common definition of the gap  $\delta_1:=\min_{\lambda_i\in S(\Lambda_1), \lambda_j\in S(\Lambda_2)}|\lambda_i-\lambda_j|$, we see that $\delta_0\leq\delta_1$. As a consequence, a $\sin\Theta$ bound that  requires  $\delta_0>0$ is likely to be weaker than one that requires $\delta_1>0$. In the literature, despite of the usage of the slightly weaker gap $\delta_0$, \eqref{eq:plug_stewart} provides the best known relation between the $\sin\Theta$ distance and the condition numbers. However, it suggests that $\mathcal{X}_1$ become unstable (0 tolerance of noise) as $\kappa_2(X_1) \rightarrow \infty$. 
A similar result has been derived in \cite{varah1970computing}. The original result in \cite{varah1970computing} was stated for both diagonalizable matrices and Jordan forms. Here to avoid distractions, we only state it for diagonalizable matrices, that is
\begin{equation}\label{eq:varah_bound}
    \|\sin\Theta(\mathcal{X}_1,\W{\mathcal{X}_1})\|\leq\frac{C_r\|\Delta A\|}{\sigma_{\min}(X)\sigma_{\min}(X_1)}.
\end{equation}
\sloppy Here $C_r$ is a constant depending on $r:= dim(\mathcal{X}_1)$ and the eigengap $\delta_1$  defined above. Recall that we required the columns of $X$ to be normalized, which leads to $1\leq \sigma_{\max}(X_1)=\kappa_2(X_1)\sigma_{\min}(X_1)\leq \sqrt{r} $. From this, we see that the bound \eqref{eq:varah_bound} is no smaller than 
\[
\frac{C_r \kappa_2(X_1) \|\Delta A\|}{\sqrt r \sigma_{\min}(X)}.
\]
Thus \eqref{eq:varah_bound} also suggests an instability of $\mathcal{X}_1$ as $\kappa_2(X_1)\rightarrow \infty$, which is quite pessimistic. In \cite{ipsen2003note}, the author proved that one can replace the absolute error $\|\Delta A\|$ in \eqref{eq:plug_stewart} by a relative error $\|A^{-k} \Delta A \W A^{-l}\|$, where $k$ and $l$ are positive numbers. However, the dependence on the condition numbers was not improved.

To summarize,  \eqref{eq:plug_stewart} is the state-of-the-art relation between the $\sin\Theta$ distance and the condition numbers. We can see that it is quite tight when $\kappa_2(X_1)$ and $\kappa_2(V_2)$ are small. In particular,  when $A$ is Hermitian,  $X$ is orthogonal with one as its condition number. Then \eqref{eq:plug_stewart} reduces to
\begin{equation}
    \|\tan\Theta(\mathcal{X}_1,\mathcal{\W X}_1)\|<\frac{2\|\Delta A\|}{\delta_0-2\|\Delta A\|},
\end{equation}which meets the tight Davis-Kahan's bound \cite{davis1970rotation} ensuring the stability of the subspace with sufficient eigengap from others. 

However, when $\kappa_2(X_1)$ is large, we show that   \eqref{eq:plug_stewart} is no longer tight through the following example, which motivates the main result of this paper. 

\begin{example} \label{eg:1} Consider the following matrix 
\[
A = \left[\begin{matrix} B & {\bf 0} \\  {\bf 0} &  1/2 \end{matrix} \right], \quad  B = \left[\begin{matrix}1 & 1   \\
\epsilon & 1 \end{matrix} \right].  
\]
Assume $\epsilon=o(1)$. Let  $X_1$ be the $3\times 2$ matrix containing the two eigenvectors of $A$ corresponding to the block $B$.  We want to find the stability of $\mathcal{X}_1=span(X_1)$.
\end{example}
Since $B$ is close to a Jordan block, $\kappa_2(X_1)$ must be large. To verify it, we first obtain the closed-form expressions for $X_1$ and $X_2$  
\[
X_1 =\frac{1}{\sqrt{1+\epsilon}}\begin{pmatrix}
1 & 1 \\
\epsilon^{\frac{1}{2}} & -\epsilon^{\frac{1}{2}} \\
0 & 0
\end{pmatrix}, \quad X_2 =\begin{pmatrix}
0 \\
0\\
1
\end{pmatrix}.
\]
Then this immediately implies $\mathcal{X}_1= span\{e_1,e_2\}$ ($e_i$, $i=1,2$ are the canonical basis vectors) and $\kappa_2 (X_1)=\epsilon^{-1/2}\gg 1$.

In addition, the gap is sufficiently large, since the eigenvalues in $\Lambda_1$ are $1\pm \epsilon^{1/2} $, which are sufficiently separated from $\Lambda_2=1/2$.

With a general perturbation $\Delta A$, there is no closed-form expression for $\W{\mathcal{X}}_1$. We therefore turn to some special perturbations and hope to use them to make our point. Let $E_{i,j}$ be the $3\times 3$ matrix whose $(i,j)$th entry equals 1 and other entries equal 0. Consider special perturbations of the form $\Delta A = \epsilon_1 E_{i,j}$ with $i,j\in\{1,2,3\}$, and assume $\epsilon_1 =o(1)$ is a different small constant than $\epsilon$. If $(i,j) \in \{(1,1),(1,2),(2,1),(2,2),(1,3),(2,3),(3,3)\}$, one can verify that $\W{\mathcal{X}}_1$ is exactly the same as $\mathcal{X}_1$, so $\|\sin\Theta(\mathcal{X}_1, \W{\mathcal{X}}_1)\|=0$. Else, $(i,j) = (3,1)$ and $(i,j) = (3,2)$. One can verify that the eigenvalues of $A$ do not change under the perturbation. If $(i,j)=(3,1)$, then the perturbed eigenvectors are
\[
\quad \W X_1=normalize\left(\left[\begin{matrix}
1 & 1 \\
\epsilon^{1/2} & -\epsilon^{1/2} \\
\frac{2\epsilon_1}{1+2\sqrt{\epsilon}} & \frac{2\epsilon_1}{1-2\sqrt{\epsilon}} 
\end{matrix} \right]\right),
\]
where $normalize$ stands for column-wise normalization. We can directly compute the distance between $X_1$ and $\W X_1$ to get
\[
\|\sin\Theta(\mathcal{X}_1, \W{\mathcal{X}}_1)\|\leq O(\epsilon_1) =O( \|\Delta A\|).
\]
Similarly, for $(i,j)=(3,2)$, the same calculation again yields $\|\sin\Theta(\mathcal{X}_1, \W{\mathcal{X}}_1)\|\leq O(\|\Delta A\|).$

Therefore, for all the special perturbations of the form $\Delta A = \epsilon_1 E_{i,j}$, we have $\|\sin\Theta(\mathcal{X}_1,\W{\mathcal{X}}_1)\| \leq O(\|\Delta A\|)$. Notice that this upper bound does not contain $\kappa_2(X_1) = \epsilon^{-1/2}$,  suggesting that the bound \eqref{eq:plug_stewart} that contains $\kappa_2(X_1)$ may be suboptimal, since $\kappa_2(X_1)$ is large. To get additional supporting evidence, we tested random perturbations and summarized the values of $\|\sin\Theta(\mathcal{X}_1, \W{\mathcal{X}}_1)\|$ in Table \ref{table:1}.
\begin{table}[H]
\centering
\begin{tabular}{ccccccc}
\hline
$\epsilon$ & 1e-2 & 1e-4 & 1e-6 & 1e-8 &  1e-10 \\ 
\hline 
Estimated by \eqref{eq:plug_stewart} &5.00e-5 & 4.08e-4 &4.00e-3 & 0.0042 &0.67\\
\hline
True $\sin\Theta$ distance &2.07e-6 & 1.99e-6 &1.99e-6 &1.99e-7 &1.99e-6\\
\hline
\end{tabular}
\caption{Comparison of the true $\sin\Theta$ distance in Example \ref{eg:1} with its upper bound computed from \eqref{eq:plug_stewart} for various values of $\epsilon$. The perturbation matrix $\Delta A$ is a realization of the random Gaussian matrix rescaled to a fixed norm $\|\Delta A\|=\epsilon_1=10^{-6}$. With this fixed $\Delta A$, we let the condition number $\kappa_2(X_1)\rightarrow \infty$ by letting $\epsilon \rightarrow 0$ in Example \ref{eg:1}. We see that the true $\sin\Theta$ distance does not vary with $\epsilon$ while the upper bound \eqref{eq:plug_stewart} does, suggesting the suboptimality of \eqref{eq:plug_stewart}.}\label{table:1}
\end{table}

In the simulation, we added a random perturbation to the matrix $A$ defined in Example \ref{eg:1} and let the $\epsilon$ in $A$ take various values. For each value of $\epsilon$, we compute the true $\sin\Theta$ distance as well as its upper bound in \eqref{eq:plug_stewart}. We observed that the true $\sin\Theta$ distance does not change much as $\epsilon\rightarrow 0$ while the upper bound blows up,  which suggests a suboptimality of the bound. More specifically, the random perturbation $\Delta A$ in this experiment was obtained by re-scaling  an i.i.d. Gaussian matrix to have a spectral norm of $10^{-6}$. Because in Example \ref{eg:1}, we took $X_1$ to be the eigenvectors of $A$ corresponding to the two largest magnitude eigenvalues, we also need to take $\W X_1$ to be the eigenvectors of $\W A$ associated with the two largest magnitude eigenvalues. The true $\sin\Theta$ distance in the table was computed by $\sqrt{1-\sigma_{min}^2\left(Q_{X_1}^*Q_{\W X_1}\right)}$, which equals $\left\|\sin\Theta(Q_{X_1},Q_{\W X_1})\right\|$.

\subsection{Contribution}
Both the theoretical argument for special perturbations and the numerical results for random perturbations suggest that for the matrix $A$ in Example \ref{eg:1}, the perturbation of $\mathcal{X}_1$ is $O(\epsilon_1)=O(\|\Delta A\|)$, which is unaffected by the large condition number $\kappa_2 (X_1) = \epsilon^{-1/2}$ as $\epsilon\rightarrow 0$. 

In this paper, we show that this phenomenon is no coincidence, and $\kappa_2(X_1)$ can indeed be removed from the previous bound \eqref{eq:plug_stewart}. Our result states that 
\[
\|\sin\Theta(\mathcal{X}_1,\W{\mathcal{X}}_1)\|\leq \kappa_2(V_2)f\left(\|A\|,\|\Delta A\|, \delta_{1},r\right),
\] where $f$ is some function of $\|A\|,\|\Delta A\|, \delta_{1}$, and $r\equiv dim(\mathcal{X}_1)$. The new bound ensures that the stability will not keep getting worse as $\kappa_2(X_1)\rightarrow \infty$.  In particular, when $A$ approaches a Jordan form, we may still stably estimate its invariant subspaces from noisy data. Note that there is a previous result \cite{varah1970computing} that guaranteed the stability of invariant subspaces for deficient matrices (which correspond to $\kappa_2(X_1)=\infty$), while our result holds for the much larger class of diagonalizable matrices with large condition numbers. The proof technique is also completely different.

The structure of the rest of the paper is organized as follows. In Section 2, we set up the notation, and our main result is presented in Section 3.
\section{Notation \& Assumptions}\label{sec:notations}
Let $A,\W A,X,\W X$ be the same as defined before Section \ref{sec:motive}. Define $V:=(X^{-1})^*$ and $\W V:=(\W X^{-1})^*$, then $V^*X=\W V^*\W{X}=I$. For any matrix $Z$, we use $Z = Q_{Z}R_{Z}$ to denote the $QR$ decompositions of $Z$, $S(Z)$ to denote the spectrum of $Z$, $\rho(Z)$ the spectral radius of $Z$, and $\#(\cdot)$ the cardinality of a set. We denote by $\|\cdot\|$ the $\ell_2$-norm of vectors and the spectral norm of matrices, by $\|\cdot\|_F$ the Frobenius norm of matrices. In addition, for two functions $f(t)$ and $g(t)$, $f=O(g)$ means that $f$ is asymptotically upper bounded by $g$ as $t\rightarrow 0$, and $f=\Omega(g)$ means that $f$ is  asymptotically lower bounded by $g$ as $t\rightarrow 0$ . Furthermore, we use $e_i$ to denote the $i$th canonical basis vector, and $B(c,r)$ to denote a disk in complex plane centered at $c$ with radius $r$. For a matrix $A$, $A^*$ is its conjugate transpose,  $\kappa_2(A)=\frac{\sigma_{\max}(A)}{\sigma_{\min}(A)}$ is the condition number, and $\sigma_{\max}(A)$ and $\sigma_{\min}(A)$ are the largest and smallest singular values of $A$, respectively. 

In addition, we assume that the spectra of $A$ and $\W A$ have gaps. Since the study of the eigenvalue perturbation is out of the scope of this paper, we simply make the existence of eigengaps as an assumption.

\textbf{Assumption 1 [Eigengap]:} Suppose $S(\Lambda_1)$ and $S(\Lambda_2)$ are well-separated in the sense of $\min_{\lambda \in S(\Lambda_1),\sigma\in S(\Lambda_2)}|\lambda-\sigma|>0 $.


The following assumption assumes that the gap still exists after perturbation. 

\textbf{Assumption 2 [Eigengap under perturbation]:}  Suppose $S(\W \Lambda_1)$ and $S(\Lambda_2)$ are well-separated with some eigengap $\delta_\lambda>0$. More explicitly,  $0<\delta_\lambda:= \min_{\lambda \in S(\W\Lambda_1),\sigma\in S(\Lambda_2)}|\lambda-\sigma| $.


To measure the distance between the original and the perturbed eigen-subspaces $\mathcal{X}_1=span(X_1)=span(Q_{X_1})$ and $\W{\mathcal{X}_1}=span(\W X_1)=span(Q_{\W X_1})$, we follow the usual definition of principle angles between two subspaces \cite{bjorck1973numerical}.  Suppose $Q,\widetilde Q \in \mathbb{C}^{n \times r}$ $(n\geq r)$ are two matrices with orthonormal columns. Let the singular values of $Q^*\widetilde Q$ be $\zeta_1\geq\zeta_2\geq\cdots\geq \zeta_r\geq 0$, then $\cos^{-1}{\zeta_i}$, $i=1,\cdots,r$ are the principal angles. Define the $\sin\Theta$ angles between $span(Q)$ and $span(\W Q)$ as
\[
\textrm{sin}\Theta(Q,\widetilde Q)=\text{diag}\left\{\sin\left(\cos^{-1}(\zeta_1)\right),\sin\left(\cos^{-1}(\zeta_2)\right),\cdots,\sin\left(\cos^{-1}(\zeta_r)\right)\right\}.
\]
The distances between the two subspaces are then characterized by $||\sin\Theta(Q,\widetilde Q)||$, where $||\cdot||$ is the spectral norm. Direct calculation gives
\begin{equation}\label{eq:sinform}
\|\sin\Theta(Q,\widetilde Q)\|=\|Q_\perp^*\widetilde Q\|=\|\widetilde Q_\perp^* Q\|,
\end{equation}
where $Q_\perp$ is the orthogonal complement of $Q$. Hence in order to bound the $\sin \Theta$ angles between two subspaces, it is sufficient to bound $\|Q_\perp^*\widetilde Q\|$ or $\|\W Q_\perp^* Q\|$.
\section{Main results}\label{sec:main}
The following theorem improves upon \eqref{eq:plug_stewart} when $\kappa_2(X_1)$ is large.
\begin{theorem}\label{thm:main} Assume $A$ and $\W A$ are diagonalizable matrices with decomposition \eqref{eq:defination} and \eqref{eq:tildeA}, and satisfying Assumption 2 with an eigengap $\delta_\lambda$. Let $r=\#\left(\mathcal{S}(\Lambda_1)\right)$ be the number of eigenvalues in $\Lambda_1$ and denote by $\W \lambda_j$, $j=1,...,r$ the $j$th diagonal element in $\W \Lambda_1$ then we have
\begin{align}\label{eq:main1}
\begin{split}
\left\|\sin \Theta(Q_{X_1},Q_{\W X_1}) \right\| & \leq  \frac{\kappa_2(V_2)\|\Delta A\|_F}{a} \prod_{j=1}^r \left(1+\frac{a}{\min\limits_{\lambda_k\in S(\Lambda_2)}|\W {\lambda}_j - \lambda_{k}|}\right)\\
&\leq \frac{\kappa_2(V_2)\|\Delta A\|_F}{a} \prod_{j=1}^r \left(1+\frac{a}{\delta_\lambda}\right), 
\end{split}
\end{align}
where $a = \|A\|+\|\Delta A\| + \rho(\Lambda_2)$, and $\rho(\Lambda_2)$ is the spectral radius of $\Lambda_2$.
\end{theorem}

\begin{remark}
Notice that \eqref{eq:main1} does not contain $\kappa_2(X_1)$, and when being applied to Example \ref{eg:1}, agrees with the numerical observation.
\end{remark}
\subsection{Tightness of the bound \eqref{eq:main1}}
At first glance, the bound in \eqref{eq:main1} contains the $r$th power of the eigengap in the denominator, which seems unusual. In this section, we demonstrate that it is actually tight for general matrices. 
\subsubsection*{The dependence  on $\delta_\lambda^r$ is tight}
We give an example showing that the dependence  on $\delta_\lambda$ is tight, i.e., the $r$th power on $\delta_\lambda$ in the upper bound can be attained by the following examples.
\begin{example}
We first present an example for $r=2$.
\[
A = \left[\begin{matrix} 1 & 0 & 0 \\ 1 &  1-\delta & 0 \\ 0 & 0 & 1-2\delta \end{matrix} \right], \quad  \Delta  A = \left[\begin{matrix}0 & 0 & 0   \\
0 & 0 & 0 \\ 0 & \epsilon & 0 \end{matrix} \right].  
\]
Here we set $\epsilon=\min\{o(1),O(\delta^2)\}$, $0<\delta<1$. We consider the perturbation of the subspace $\mathcal{X}_1$ spanned by the largest two eigenvectors of $A$, so $r:=dim(\mathcal{X}_1)=2$. It is immediate to verify that $\mathcal{X}_1 = span(e_1,e_2)$ ($e_i$ is the $i$th canonical basis), $V_2=e_3$, and the eigengap between the largest two and the smallest eigenvalues is $\delta_\lambda=\delta$. With the given perturbation, one can verify that $\lambda_i=\W \lambda_i$, for $i=1,2,3$. The perturbed subspace can also be calculated   
\[
\W{\mathcal{X}}_1 = span\left\{\left[\begin{matrix} 1   \\  \frac{1}{\delta}   \\ \frac{\epsilon}{2\delta^2}   \end{matrix} \right],\left[\begin{matrix} 0  \\    1  \\  \frac{\epsilon}{\delta} \end{matrix} \right]\right\}.
\]
In $\W{\mathcal{X}}_1$, we pick a special vector  that has a large angle with the original subspace ${\mathcal{X}}_1$: pick $x = \left[1,0,-\frac{\epsilon}{2\delta^2}\right]^T \in \W{\mathcal{X}}_1$. One can verify that the sine of the angle between $x$ and $\mathcal{X}_1$ reaches $ \Omega\left(\frac{\epsilon}{\delta^2}\right)= \Omega\left(\kappa_2(V_2)\frac{\|\Delta A\|_F}{\delta_\lambda^r}\right)$, since $\kappa_2(V_2)=1$, $r=2$ and $\|\Delta A\|_F=\epsilon$  (The Big $\Omega$ notation was defined in Section \ref{sec:notations}). Since $\|\sin\Theta({\mathcal{X}}_1,\W{\mathcal{X}}_1)\|$ is at least as large as the $\sin\Theta$ angle between $x$ and $\mathcal{X}_1$, then $\|\sin\Theta({\mathcal{X}}_1,\W{\mathcal{X}}_1) \|\geq \Omega\left(\kappa_2(V_2)\frac{\|\Delta A\|_F}{\delta_\lambda^r}\right)$, and therefore the $r$th power of the eigengap is reached for the case $r=2$.

The above construction can be easily generalized to any $r>2$. 
\end{example}
\begin{example}
\[
A = I_{r+1}-\left[\begin{matrix}
0 &  & & & &\\
-1 &  \delta &  & & &\\
 & -1 & 2\delta & & &\\ 
&  & \ddots & \ddots & & \\
& & & -1&(r-1)\delta &  \\
 & & & &0 & r\delta\\
\end{matrix} \right],
\quad  \Delta  A = \left[\begin{matrix}0 & \cdots&\cdots&\cdots & 0   \\
\vdots &\cdots&\cdots& \cdots & \vdots \\ 0 &\cdots& 0&\epsilon & 0 \end{matrix} \right].  
\]
Here except for the $\epsilon$, all entries of the perturbation matrix $\Delta A$ are zero. Again, we set $\epsilon=\min\{o(1),O(\delta^r)\}$, $0<\delta<1$. Consider the perturbation of the subspace spanned by the eigenvectors associated with the largest $r$ eigenvalues, which is the subspace $\mathcal{X}_1=span\left(e_1,e_2,...,e_r\right)$. One can easily verify that $V_2=e_{r+1}$, and the eigengap is still $\delta_\lambda=\delta$. We can again pick a special vector in the perturbed invariant subspace $\W{\mathcal{X}}_1$: $x=\left[1,0,0,\cdots,(-1)^{r-1}\frac{\epsilon}{r!\delta^r}\right]^T$. The angle between $x$ and $\mathcal{X}_1$ then reaches $\Omega\left(\frac{\epsilon}{\delta^r}\right)=\Omega\left(\kappa_2(V_2)\frac{\|\Delta A\|_F}{\delta_\lambda^r}\right)$. Therefore the power $r$ on the eigengap in the denominator of \eqref{eq:main1} is attained.
\end{example}
\subsubsection*{The $\kappa_2(V_2)$ in \eqref{eq:main1} cannot be removed}
The bound in Theorem \ref{thm:main} successfully gets rid of $\kappa_2(X_1)$. The next example shows that we cannot further remove $\kappa_2(V_2)$ from it.
\begin{example} We first consider a matrix of three dimensions.
\[
A = \left[\begin{matrix} 1+\delta & 0 & 0 \\  0 &  1 & 0 \\ 0 & \frac{1}{2} & 1-\delta_1 \end{matrix} \right], \quad  \Delta  A = \left[\begin{matrix}0 & 0 & 0   \\
\epsilon & 0 & 0 \\ 0 & 0 & 0 \end{matrix} \right],  
\]
where $\epsilon=\min\{o(1),O(\delta^2)\}$ and $0<\delta,\delta_1\ll 1$. Consider $\mathcal{X}_1$ to be the 1-dimensional subspace spanned by the eigenvector $\left[1,0,0\right]^T$ of $A$ associated with the largest eigenvalue, so $r=1$. Under the given perturbation, one can check that the perturbed eigenvector associated with the largest eigenvalue is $\left[1,\frac{\epsilon}{\delta},\frac{\epsilon}{2\delta(\delta+\delta_1)}\right]^T$, so $\W{\mathcal{X}}_1 = span\left(\left[1,\frac{\epsilon}{\delta}, \frac{\epsilon}{2\delta(\delta+\delta_1)}\right]^T\right)$. As a result, 
\[
\|\sin\Theta(\mathcal{X}_1,\W{\mathcal{X}}_1)\|= \Omega\left( \frac{\epsilon}{\delta (\delta+\delta_1)}\right).
\]
It is also immediate that $\kappa_2(V_2)= \Omega\left( \frac{1}{\delta_1}\right)$, $\delta_\lambda=\delta$, $\|\Delta A\|_F=\epsilon$ and $\|A\|=O(1)$. Plugging these into the bound \eqref{eq:main1}, we get that the bound is $O\left( \frac{\epsilon}{\delta \delta_1}\right)$. We can see that this bound no smaller than the actual $\sin\Theta$ angle. But if $\kappa_2(V_2)$ was absent from the bound, then  the bound \eqref{eq:main1} would only be $O\left( \frac{\epsilon}{\delta}\right)$, which is no longer enough to bound the actual $\sin\Theta$ angle. Therefore the appearance of $\kappa_2(V_2)$ is necessary. The same idea allows us to construct such examples for any dimension. Specifically, for any dimension $n$, we can define 
\[
A = \left[\begin{matrix} 1+\delta & \ & \ &\ \\  \ &  1 & \ &\ \\ \ & \frac{1}{2} & 1-\delta_1 &\ \\
\ &\ &\ &(1-2\delta_1)I_{n-3}
\end{matrix} \right], \quad  \Delta  A = \left[\begin{matrix} 0& \cdots & 0  \\ 
\epsilon & \vdots &  \vdots\\ 
0& \vdots & \vdots \\
\vdots& \vdots & \vdots\\
0 &\cdots & 0
\end{matrix} \right].  
\]
\end{example}
Here the perturbation matrix $\Delta A$ only contains one nonzero element at its $(2,1)$th entry. Let $\mathcal{X}_1$ to be the subspace spanned by the first eigenvector, so again, $r=1$. Direct calculation gives that $\mathcal{X}_1=span(e_1)$, the eigenvector of $\W A$ associated with the largest eigenvalue is $\left[1,\frac{\epsilon}{\delta},\frac{\epsilon}{2\delta(\delta+\delta_1)},0,...,0\right]^T$, hence $\W{\mathcal{X}_1}=span\left(\left[1,\frac{\epsilon}{\delta},\frac{\epsilon}{2\delta(\delta+\delta_1)},0,...,0\right]^T\right)$. We also have $\kappa_2(V_2)=\Omega\left(\frac{1}{\delta_1}\right)$, $\|\Delta A\|_F=\epsilon$, and $\|A\|=O(1)$. Then same discussion implies that the appearance of $\kappa_2(V_2)$ in the upper bound is necessary.
\subsubsection*{A remark on the $a$ in the bound}
Observe that the bound in \eqref{eq:main1} also contains $a$, which is essentially the spectral norm of $A$. We argue that the presence of $a$ is necessary as it ensures that the bound is scaling invariant. More specifically, replacing $A$ and $\W A$  by $tA$ and $t\W A$ with any scalar $t\neq 0$, we see that our bound in \eqref{eq:main1} does not change, which matches the fact that the angle between the original and the perturbed subspaces is invariant to a universal scaling.
\subsection{Proof of the theorem}
In order to prove Theorem \ref{thm:main}, we first state an equivalent expression of $\|\sin\Theta(Q_{X_1},Q_{\W X_1})\|$, 
\begin{equation*}
\|\sin\Theta(Q_{X_1},Q_{\W X_1})\|=\|Q_{V_2}^*Q_{\widetilde X_1}\|.
\end{equation*}
Here $\W X_1=Q_{\W X_1}R_{\W X_1},V_2=Q_{V_2}R_{V_2}$ are the QR decompositions of $\W X_1$ and $V_2$, respectively. The proof is based on the following lemma in \cite{li1994perturbations}.
\begin{lemma}[Lemma 2.1 in \cite{li1994perturbations}] \label{lm:sin}
Let $U_1,\widetilde U_1\in\mathbb{C}^{n,r}\ (1\leq r\leq n-1)$ with $U_1^* U_1={\widetilde U_1}^*\widetilde U_1=I$, and let $\mathcal{X}_1=span(U_1)$ and $\widetilde{\mathcal{X}}_1=span(\widetilde U_1)$. If $\widetilde U=[\widetilde U_1,\widetilde U_2]$ is a unitary matrix, then $\|\sin\Theta(\mathcal{X}_1,\widetilde{\mathcal{X}}_1)\|=\|{\widetilde U_2}^* U_1\|$.
\end{lemma}
In Lemma \ref{lm:sin}, let $U_1=Q_{\widetilde X_1}, \W U_1=Q_{X_1}, \W U= [Q_{X_1},Q_{V_2}]$. Noticing that $V^*X=I$, we can verify that $\widetilde U^*\widetilde U=I$. Therefore, it holds that $\|\sin\Theta(Q_{X_1},Q_{\W X_1})\|=\|Q_{V_2}^*Q_{\widetilde X_1}\|$.

The next Lemma gives an equivalent expression of $Q_{V_2}^*Q_{\W X_1}$.
\begin{lemma}\label{lm:equation}
Using the notations in Section \ref{sec:notations}, it holds that
\begin{equation}\label{eq:lemma}
    Q_{V_2}^*Q_{\W X_1} = (R_{V_2}^{-1})^*\left(F\circ\left(V_2^*\Delta A\W X_1\right)\right)R_{\W X_1}^{-1},
\end{equation}
where $\circ$ denotes the Hadamard product, and $F \in \mathbb{C}^{n-r,r}$ is defined as $F_{i,j} = \left(\W \lambda_j- \lambda_{r+i}\right)^{-1}$, for $i=1,...,n-r, j=1,...,r$. 
\end{lemma}
Lemma \ref{lm:equation} has been implicitly derived in \cite{li1998spectral}, here for completeness, we provide its proof here. An alternative proof using complex analysis can be found in the appendix.
\begin{proof}
Since $X^{-1}AX=\Lambda$ and $\ {\W X}^{-1} \W A\W X=\W\Lambda$,
then 
\begin{align*}
    X^{-1}\Delta A\W X=-X^{-1}(A-\W A)\W X=-\Lambda X^{-1}\W X+X^{-1}\W X\W\Lambda.
\end{align*}
Consider the $(n-r) \times r$ block in the lower-left corner of this equation, we have
\begin{equation}\label{eq:hadamard}
V_2^*\Delta A \W X_1 = \bar F\circ\left(V_2^* \W X_1\right),
\end{equation}
where $\bar F_{i,j}=\W \lambda_j-\lambda_{r+i},\ 1\leq i\leq n-r,\ 1\leq j\leq r$. Let $F = 1/\bar F$, where the division is carried out elementwise, \eqref{eq:hadamard} becomes
\[
V_2^*\W X_1 = F\circ\left(V_2^*\Delta A\W X_1\right).
\]
Last, we replace the $V_2^*$  and $\W X_1$ on the left-hand side with their QR decomposition and move the $R$ factors to the right-hand sides to obtain the equation in the statement of this lemma.
\end{proof}
Denoting $M=\left(F\circ\left(V_2^*\Delta A\W X_1\right)\right)R_{\W X_1}^{-1}$, by Lemma \ref{lm:equation}, we have
\begin{equation}\label{eq:sin_distance}
    \left\|\sin\Theta(Q_{X_1},Q_{\W X_1})\right\|=\left\|Q_{V_2}^*Q_{\W X_1}\right\|=\left\|(R_{V_2}^{-1})^* M\right\|\leq \frac{1}{\sigma_{\min} (R_{V_2})}\left\|M\right\|.
\end{equation}
In order to bound $\|M\|$, we establish the following lemma for an equivalent expression of $M$.
\begin{lemma} \label{lm:m} Denote $M=\left[m_1,m_2,...,m_{n-r}\right]^*$. Then for $1\leq i\leq n-r$, the $i$th row of $M$ can be expressed as
\begin{equation}\label{eq:mi}
m_i^* = \frac{1}{(-1)^{r+1}\sigma_{r}}  V_{2,i}^*\Delta A  \left(\hat{A}^{r-1}- \sigma_1 \hat{A}^{r-2}+ \sigma_2\hat{A}^{r-3}-\cdots+(-1)^{r-1}\sigma_{r-1}I_n\right)Q_{\W {X}_1},
\end{equation}
where $V_{2,i}$ is the $i$th column of $V_2$, $\hat{A} = \W A - \lambda_{i+r}I_n$, $I_n$ is the identity matrix of size $n$, $\lambda_{i+r}$ is the $i$th diagonal element in $\Lambda_2$, $\sigma_k$ is the homogeneous symmetric polynomial of order $k$ in $r$ variables, that is
\[
\sigma_k = \sum_{1\leq i_1<i_2,...,i_{k-1}<i_k\leq r} \tl_{i_1}\tl_{i_2}\cdots \tl_{i_k},\ k=1,2,...,r.
\]
 Here $\hat{\lambda}_j := \W \lambda_j -\lambda_{i+r}$, $j=1,...,r$, and $\W \lambda_j$ the $j$th diagonal element in $\W \Lambda_1$. By Assumption 2, $\hat{\lambda}_j \neq 0$. 
\end{lemma}
\begin{proof}[Proof of Lemma \ref{lm:m}]
Let $b_i^*$ be the $i$th row of the $(n-r) \times r$ matrix $V_2^*\Delta A$. Then by Lemma \ref{lm:equation} , the $i$th row of $M$ is 
\begin{equation}\label{eq:m}
m_i^* = \left(b_i^*\W X_1\right) \left[\begin{matrix} \frac{1}{\W \lambda_1-\lambda_{i+r}} &   &  \\ & \ddots &    \\ & &  \frac{1}{\W \lambda_r-\lambda_{i+r}} \end{matrix}\right] R_{\W X_1}^{-1}.
\end{equation}
Let $y$ be an arbitrary unit vector, and let $\hat{\lambda}_j := \W \lambda_j -\lambda_{i+r}$, for $j=1,...,r$. In addition, define $p = R_{\W X_1} ^{-1} y$, which means 
\begin{equation}\label{eq:p} 
\|\W X_1 p \| = \left\|Q_{\W X_1} y\right\| = 1.
\end{equation} Then \eqref{eq:m} yields
\begin{equation}\label{eq:m_y}
m_i^* y = b_i^* \left(\sum_{j=1}^r \frac{1}{\hat{\lambda}_j }p_j \W x_j\right),
\end{equation}
where for $1\leq j\leq r$, $\W x_j$ is the $j$th column of $\W X_1$. Next, we derive an equivalent representation of the summation in \eqref{eq:m_y} with the help of the characteristic polynomial. Define $\hat A=\W A-\lambda_{i+r}I$, and define its characteristic polynomial 
\begin{equation*}
    q(z)=\left(z-\hat \lambda_1\right)\left(z-\hat \lambda_2\right)\cdots\left(z-\hat\lambda_r\right).
\end{equation*}
Expanding $q(z)$ leads to
\begin{equation}\label{eq:charac}
q(z) = z^r- \sigma_1 z^{r-1}+ \sigma_2z^{r-2}-\cdots+(-1)^{r-1}\sigma_{r-1}z+(-1)^{r}\sigma_{r},
\end{equation}
where $\sigma_k$, $k=1,...,r$ are the homogeneous symmetric polynomials of order $k$ in $r$ variables, that is
\[
\sigma_k = \sum_{1\leq i_1<i_2,...,i_{k-1}<i_k\leq r} \tl_{i_1}\tl_{i_2}\cdots \tl_{i_k}.
\]
Notice that $\W X_1$ is also invariant to $\hat A$. Since
\[
\left(\hat A-\hat\lambda_jI\right)\W X_1=\left(\W A-\W\lambda_j\right)\W X_1=\W X_1\left(\W\Lambda_1-\W\lambda_j I_n\right),\ 1\leq j\leq r,
\]
then $q(\hat A)\W X_1=\W X_1\left(\W\Lambda_1-\W\lambda_1 I\right)\left(\W\Lambda_1-\W\lambda_2 I\right)\cdots\left(\W\Lambda_1-\W\lambda_r I\right)=0$. This means for any $c\in\mathbb{C}^r$, we have
\[
0=q\big(\hat A\big)\W X_1c=\left(\hat A^r-\sigma_1\hat A^{r-1}-\cdots+(-1)^{r-1}\sigma_{r-1}\hat A+(-1)^r\sigma_rI_n\right)\W X_1c.
\]
Let us move the last term in the right-hand side to the left and for the terms left on the right, pull one $\hat{A}$ out of the bracket, 
\begin{equation}\label{eq:inter_equa}
(-1)^{r+1}\sigma_{r}\tilde{X}_1c = \left(\hat{A}^{r-1}- \sigma_1 \hat{A}^{r-2}+ \sigma_2\hat{A}^{r-3}-\cdots+(-1)^{r-1}\sigma_{r-1}I_n\right)\hat{A}\W {X}_1c .
\end{equation}
Now let us take $c$ to be the vector consisting of $c_j=p_j/\tl_j$, for $j=1,...,r$, then $\hat{A}\W {X}_1c  = \W {X}_1p$ and $\W {X}_1c=\sum_{j=1}^r \frac{1}{\tl_j}\W {x}_jp_j$. Plugging these two relations into \eqref{eq:inter_equa}, we get
\[
(-1)^{r+1}\sigma_{r}\left(\sum_{j=1}^r \frac{1}{\tl_j}p_j \W {x}_j\right) = \left(\hat{A}^{r-1}- \sigma_1 \hat{A}^{r-2}+ \sigma_2\hat{A}^{r-3}-\cdots+(-1)^{r-1}\sigma_{r-1}I_n\right)\tilde{X}_1p,
\]
or equivalently,
\begin{align}\label{eq:main_identity}
\sum_{j=1}^r \frac{1}{\tl_j}p_j\tilde{x}_j & = \frac{ \left(\hat{A}^{r-1}- \sigma_1 \hat{A}^{r-2}+ \sigma_2\hat{A}^{r-3}-\cdots+(-1)^{r-1}\sigma_{r-1}I_n\right)\tilde{X}_1p}{(-1)^{r+1}\sigma_{r}}. 
\end{align}
Plugging this back to the formula for $m_i^*y$ \eqref{eq:m_y}, we get
\[
m_i^*y = \frac{1}{(-1)^{r+1}\sigma_{r}}  V_{2,i}^*\Delta A \left(\hat{A}^{r-1}- \sigma_1 \hat{A}^{r-2}+ \sigma_2\hat{A}^{r-3}-\cdots+(-1)^{r-1}\sigma_{r-1}I_n\right)Q_{\W X_1}y.
\]
The equation above holds for arbitrary $y\in\mathbb{C}^r$, hence \eqref{eq:mi} holds.
\end{proof}
Next, we prove Theorem \ref{thm:main}.
\begin{proof}[Proof of Theorem \ref{thm:main}]
Denote $b_i^*=V_{2,i}^*\Delta A$, by Lemma \ref{lm:m} we have
\begin{align*}
    \|m_i^*\|&=\left\|b_i^*\frac{\hat{A}^{r-1}- \sigma_1 \hat{A}^{r-2}+ \sigma_2\hat{A}^{r-3}-\cdots+(-1)^{r-1}\sigma_{r-1}I_n}{\sigma_r}Q_{\W {X}_1}\right\|\\
    &\leq \|b_i^*\|\frac{\left\|\hat{A}^{r-1}- \sigma_1 \hat{A}^{r-2}+ \sigma_2\hat{A}^{r-3}-\cdots+(-1)^{r-1}\sigma_{r-1}I_n\right\|}{|\sigma_r|}\\
    &\leq\|b_i^*\|\frac{\|\hat A\|^{r-1}+|\sigma_1|\|\hat A\|^{r-2}+|\sigma_2|\|\hat A\|^{r-3}+\cdots+|\sigma_{r-1}|}{|\sigma_r|}.
\end{align*}
Notice that 
\[
|\sigma_k| = \left|\sum_{1\leq i_1<i_2,...,i_{k-1}<i_k\leq r} \tl_{i_1}\tl_{i_2}\cdots \tl_{i_k}\right|\leq \sum_{1\leq i_1<i_2,...,i_{k-1}<i_k\leq r} |\tl_{i_1}|\cdot|\tl_{i_2}|\cdots|\tl_{i_k}|.
\]
Define an auxiliary function 
\begin{align*}
    \bar q(z)=\left(z+|\hat\lambda_{1}|\right)\left(z+|\hat\lambda_{2}|\right)\cdots\left(z+|\hat\lambda_{r}|\right)=z^r+\bar\sigma_1 z^{r-1}+\bar\sigma_2 z^{r-2}+\cdots+\bar\sigma_r.
\end{align*}
Here $\bar\sigma_k=\sum_{1\leq i_1<i_2<\cdots<i_{k-1}<i_k\leq r}|\hat\lambda_{i_1}|\cdot|\hat\lambda_{i_2}|\cdots|\hat\lambda_{i_k}|\geq|\sigma_k|$, for $1\leq k\leq r$, and $\bar\sigma_r=|\hat\lambda_{1}|\cdot|\hat\lambda_{2}|\cdots|\hat\lambda_{r}|=|\sigma_r|$. Then
\begin{align*}
    \|m_i\|&\leq \|b_i^*\|\frac{\|\hat A\|^{r-1}+\bar\sigma_1\|\hat A\|^{r-2}+\bar\sigma_2\|\hat A\|^{r-3}+\cdots+\bar\sigma_{r-1}}{\bar\sigma_r}\\
    &\leq\|b_i^*\|\frac{\bar q(a)-\bar\sigma_r}{a\bar\sigma_r}\\
    &=\frac{\|b_i\|}{a}\left(\prod_{j=1}^r\left(1+\frac{a}{|\hat \lambda_j|}\right)-1\right)\\
    &\leq\frac{\|b_i\|}{a}\prod_{j=1}^r\left(1+\frac{a}{\min\limits_{\lambda_k\in S(\Lambda_2)}|\W\lambda_j-\lambda_k|}\right),
\end{align*}
where $a=\|A\|+\|\Delta A\|+\rho(\Lambda_2)\geq\|\hat A\|$. Combining the bounds for all $1\leq i\leq n-r$ leads to 
\begin{align*}
\|M\|&\leq \frac{\left\|V_2^*\Delta A\right\|_F}{a} \prod_{j=1}^r\left(1+\frac{a}{\min\limits_{\lambda_k\in S(\Lambda_2)}|\W\lambda_j-\lambda_k|}\right)\leq \frac{\|V_2\|\|\Delta A\|_F}{a} \prod_{j=1}^r\left(1+\frac{a}{\min\limits_{\lambda_k\in S(\Lambda_2)}|\W\lambda_j-\lambda_k|}\right).
\end{align*}
By \eqref{eq:sin_distance}, we further obtain \eqref{eq:main1}.
\end{proof}
\normalem
\bibliographystyle{abbrv}
\bibliography{refs}
\section{Appendix}
In this appendix, we show the derivation of \eqref{eq:plug_stewart} and provide an alternative proof of Lemma \ref{lm:equation}. We start with the definition of the $\tan\Theta$ angles between two subapces $\mathcal{X}$ and $\W{\mathcal{X}}$. Similar to the definition of the $\sin\Theta$ angles, suppose $Q,\W Q\in\mathbb{C}^{n,r}$ are two orthogonal matrices that spanned the two subspaces. Then the cosines of the principal angles between the subspaces are the singular values $\zeta_1\geq\zeta_2\geq\cdots\geq \zeta_r\geq 0$ of $Q^*\widetilde Q$. The $\tan\Theta$ angles are then defined to be the tangents of the principal angles,  $\tan(\cos^{-1}{\zeta_i})$, $i=1,\cdots,r$. The matrix that holds all the $\tan\Theta$ angles  is
\[
\textrm{tan}\Theta(Q,\widetilde Q)=\text{diag}\{\tan(\cos^{-1}(\zeta_1)),\tan(\cos^{-1}(\zeta_2)),\cdots,\tan(\cos^{-1}(\zeta_r))\}.
\]
\subsection{Derivation of \eqref{eq:plug_stewart}}
We state a slightly simplified version of 
the classical result by Stewart \cite{stewart1971error}. 
\begin{proposition}[simplified version of Theorem 4.1 \cite{stewart1971error}]\label{thm:stewart}
Provided that 
\begin{equation}\label{eq:condition}
\|\Delta A\|(\|A\|+\|\Delta A\|)<\frac{1}{4}\left(sep(Q_{X_1}^*AQ_{X_1},Q_{V_2}^*AQ_{V_2})-2\|\Delta A\|\right)^2,
\end{equation}
 the following error bound holds
\begin{equation}\label{eq:stewart1}
\|\tan\Theta(Q_{X_1},Q_{\W X_1})\|<2\frac{\|\Delta A\|}{sep(Q_{X_1}^*AQ_{X_1},Q_{V_2}^*AQ_{V_2})-2\|\Delta A\|},
\end{equation}
where for any pair of matrices $L_1$, $L_2$, $sep(L_1,L_2):=\inf_{\|T\|=1} \|TL_1-L_2T\|$.
 \end{proposition}
Since $\|\tan\Theta\|$ is larger than $\|\sin\Theta\|$, this also gives a sin$\Theta$ bound. 

Direct calculation gives that $sep(Q_{X_1}^*AQ_{X_1},Q_{V_2}^*AQ_{V_2})=sep(R_{X_1}\Lambda_1R_{X_1}^{-1},R_{V_2}\Lambda_2R_{V_2}^{-1})$. If we define the eigen-gap
\[
\delta_0 = \max_{t_0\in\mathbb{C}}\left\{\max\left\{\min_{\lambda\in S(\Lambda_1)}|\lambda-t_0|-\max_{\mu\in S(\Lambda_2)}|\mu-t_0|,\min_{\lambda\in S(\Lambda_2)}|\lambda-t_0|-\max_{\mu\in S(\Lambda_1)}|\mu-t_0|\right\}\right\},
\]
then a lower bound for $sep(R_{X_1}\Lambda_1R_{X_1}^{-1},R_{V_2}\Lambda_2R_{V_2}^{-1})$ can be derived by the following inequality 
\begin{equation}\label{eq:sep_bound}
sep(R_{X_1}\Lambda_1R_{X_1}^{-1},R_{V_2}\Lambda_2R_{V_2}^{-1})\geq \frac{sep(\Lambda_1,\Lambda_2)}{\kappa_2(R_{X_1})\kappa_2(R_{V_2})}\geq\frac{\delta_0}{\kappa_2(X_1)\kappa_2(V_2)},
\end{equation}
 where the first inequality is from Chapter V in \cite{stewart1990matrix}. In order to see the second inequality, it is sufficient to show $sep(\Lambda_1,\Lambda_2)\geq\delta_0$. 
By the definition of $sep$,
\begin{align*}
sep(\Lambda_1,\Lambda_2)&=\inf_{\|T\|=1} \|T\Lambda_1-\Lambda_2T\|\\
&=\inf_{\|T\|=1}\|T(\Lambda_1-t_0I)-(\Lambda_2-t_0I)T\|\\
&\geq \inf_{\|T\|=1}\left\{\min_{\lambda\in S(\Lambda_1)}|\lambda-t_0|\|T\|-\max_{\mu\in S(\Lambda_2)}|\mu-t_0|\|T\|\right\}\\
&=\min_{\lambda\in S(\Lambda_1)}|\lambda-t_0|-\max_{\mu\in S(\Lambda_2)}|\mu-t_0|.
\end{align*}
Similarly,
\[
sep(\Lambda_1,\Lambda_2) \geq \min_{\lambda\in S(\Lambda_2)}|\lambda-t_0|-\max_{\mu\in S(\Lambda_1)}|\mu-t_0|.
\]
Hence we have $sep(\Lambda_1,\Lambda_2)\geq\delta_0$, thus \eqref{eq:sep_bound} holds.

Provided that $\delta_0>2\|\Delta A\|$, plugging \eqref{eq:sep_bound} into \eqref{eq:stewart1} leads to
\begin{equation}\label{eq:plug_stewart1}
    \|\tan\Theta(Q_{X_1},Q_{\W X_1})\|<\frac{2\kappa_2(X_1)\kappa_2(V_2)\|\Delta A\|}{\delta_\lambda-2\kappa_2(X_1)\kappa_2(V_2)\|\Delta A\|}.
\end{equation}
\subsection{Alternative proof of Lemma \ref{lm:equation} using complex analysis}
Assume $(S(\Lambda_1)\cup S(\W\Lambda_1))\cap (S(\Lambda_2)\cup S(\W \Lambda_2))=\emptyset$, then there always exists a positively oriented simple closed curve $\Gamma$ in the complex plane enclosing  
the eigenvalues in $\Lambda_1$ and $\W\Lambda_1$ while leaving those in $\Lambda_2$ and $\W\Lambda_2$ outside. It has been shown in \cite{kato2013perturbation} that $P_{Q_{X_1}}=\frac{1}{2\pi i}\int_\Gamma (\lambda I-A)^{-1}d\lambda$, where $P_{Q_{X_1}}=Q_{X_1}Q_{X_1}^*$ is the projector matrix onto the subspace spanned by columns in $Q_{X_1}$. Similarly, $P_{Q_{\W X_1}}=Q_{\W X_1} Q_{\W X_1}^*=\frac{1}{2\pi i}\int_\Gamma (\lambda I-\W A)^{-1}d\lambda$, then we have
\begin{align*}
    Q_{V_2}^*Q_{\W X_1}&=Q_{V_2}^*(P_{Q_{X_1}}-P_{Q_{\W X_1}})Q_{\W X_1}\\
    &=\frac{1}{2\pi i}Q_{V_2}^*\left(\int_{\Gamma}\left((\lambda I-A)^{-1}-(\lambda I-\W A)^{-1}\right)d\lambda\right) Q_{\W X_1}\\
    &=-\frac{1}{2\pi i}\int_\Gamma Q_{V_2}^*(\lambda I-A)^{-1}\Delta A(\lambda I-\W A)^{-1}Q_{\W X_1}d\lambda\\
    &=-\frac{1}{2\pi i}\int_\Gamma Q_{V_2}^*X(\lambda I-\Lambda)^{-1}X^{-1}\Delta A\W X(\lambda I-\W\Lambda)^{-1}\W X^{-1}Q_{\W X_1}d\lambda\\
    &=-\frac{1}{2\pi i}Q_{V_2}^*X_2\left(\int_\Gamma (\lambda I-\Lambda_2)^{-1}V_2^*\Delta A\W X_1(\lambda I-\W\Lambda_1)^{-1}d\lambda\right) \W V_1^*Q_{\W X_1}\\
    &=-(R_{V_2}^{-1})^* \underbrace{\frac{1}{2\pi i}\left(\int_\Gamma (\lambda I-\Lambda_2)^{-1}V_2^*\Delta A\W X_1(\lambda I-\W\Lambda_1)^{-1}d\lambda\right)}_{G}R_{\W X_1}^{-1},
\end{align*}
where the second to last equality used the fact that $V^*X=I,\ Q_{V_2}^*X_1=0$.
 The contour integral $G=\frac{1}{2\pi i}\int_\Gamma (\lambda I-\Lambda_2)^{-1}V_2^*\Delta A\W X_1(\lambda I-\W\Lambda_1)^{-1}d\lambda$ has poles at $\W\lambda_j,1\leq j\leq r$. 
Hence the $(i,j)$th entry of $G$ can be computed by the Cauchy's  Residue Theorem as 
\begin{align*}
G_{ij}&=\frac{1}{2\pi i}\int_\Gamma
(\lambda-\W\lambda_j)^{-1}(\lambda-\lambda_{i+r})^{-1}(V_2^*\Delta A\W X_1)_{ij}d\lambda=(\W\lambda_j-\lambda_{i+r})^{-1}(V_2^*\Delta A \W X_1)_{ij}.
\end{align*}
Plugging this back into the expression of $Q^*_{V_2}Q_{\W X_1}$ gives \eqref{eq:lemma}.

\end{document}